\documentclass[psamsfonts]{amsart}

\usepackage{amsthm,amssymb,amsmath,amscd}

\theoremstyle{plain}
\newtheorem{thm}{Theorem}[section]
\newtheorem{lem}[thm]{Lemma}
\newtheorem{prop}[thm]{Proposition}

\theoremstyle{definition}
\newtheorem*{defn}{Definition}

\theoremstyle{remark}

\newtheorem*{ack}{Acknowledgement}

\newcommand{\newoperator}[2]{\DeclareMathOperator{#1}{#2}}
\newoperator{\spn}{span}
\newoperator{\graph}{graph}
\newoperator{\im}{Im}
\newoperator{\a.e.}{a.e.}
\newoperator{\diag}{diag}
\newoperator{\Aut}{Aut}

\title[C$^*$-algebras generated by composition operators]
{C$^*$-algebras generated by multiplication operators and composition operators with rational symbol}
\author{Hiroyasu Hamada}

\address{National Institute of Technology, Sasebo College, 
Okishin, Sasebo, Nagasaki, 857-1193, Japan.}
\email{h-hamada@sasebo.ac.jp}

\keywords{composition operator, multiplication operator,
Frobenius-Perron operator, C$^*$-algebra, complex dynamical system}

\subjclass[2010]{Primary 46L55, 47B33; Secondary 37F10, 46L08}

\begin{document}

\begin{abstract}
Let $R$ be a rational function of degree at least two,
let $J_R$ be the Julia set of $R$ and let $\mu^L$ be the Lyubich measure
of $R$.
We study the C$^*$-algebra $\mathcal{MC}_R$
generated by
all multiplication operators by continuous functions in $C(J_R)$
and the composition operator $C_R$ induced by $R$
on $L^2(J_R, \mu^L)$.
We show that the C$^*$-algebra $\mathcal{MC}_R$ is isomorphic to
the C$^*$-algebra $\mathcal{O}_R (J_R)$ associated with the complex dynamical
system $\{R^{\circ n} \}_{n=1} ^\infty$.
\end{abstract}

\maketitle

\section{Introduction}

Let $\mathbb{D}$ be the open unit disk in the complex plane and
$H^2 (\mathbb{D})$ the Hardy space of
analytic functions whose power series have square-summable
coefficients.
For an analytic self-map
$\varphi$ on the unit disk $\mathbb{D}$,
the composition operator
$C_{\varphi}$ on the Hardy space $H^2 (\mathbb{D})$
is defined by $C_{\varphi} g = g \circ \varphi$ for $g \in H^2 (\mathbb{D})$.
Let $\mathbb{T}$ be the unit circle in the complex plane
and $L^2(\mathbb{T})$ the square integrable measurable functions on
$\mathbb{T}$ with respect to the normalized Lebesgue measure.
The Hardy space $H^2(\mathbb{D})$ can be identified as the closed subspace of
$L^2(\mathbb{T})$ consisting of the functions whose negative
Fourier coefficients vanish.
Let  $P_{H^2}$ be the projection from $L^2(\mathbb{T})$ onto the Hardy space $H^2(\mathbb{D})$.
For $a \in L ^{\infty} (\mathbb{T})$,
the Toeplitz operator $T_a$ on the Hardy space $H^2(\mathbb{D})$
is defined by $T_a f = P_{H^2} af$ for $f \in H^2(\mathbb{D})$.
Recently several authors considered C$^*$-algebras generated by
composition operators (and Toeplitz operators).
Most of their studies have focused
on composition operators induced by linear fractional maps (\cite{J1, J2, KMM1, KMM3, KMM2, Pa, Q, Q2, SA}).



There are some studies about C$^*$-algebras generated by
composition operators and Toeplitz operators for finite Blaschke products.
Finite Blaschke products are examples of rational functions.
For an analytic self-map
$\varphi$ on the unit disk $\mathbb{D}$,
we denote by
$\mathcal{TC}_\varphi$ the Toeplitz-composition C$^*$-algebra
generated by both the composition operator $C_{\varphi}$ and the Toeplitz
operator $T_z$. Its quotient algebra by the ideal
$\mathcal{K}$ of the compact operators is denoted by
$\mathcal{OC}_{\varphi}$.
Let $R$ be a finite Blaschke product of degree at least two with
$R(0) = 0$.
Watatani and the author \cite{HW}
proved that the quotient algebra $\mathcal{OC}_{R}$
is isomorphic to the C$^*$-algebra $\mathcal{O}_R(J_R)$ associated
with the complex dynamical system introduced in \cite{KW}.
In \cite{H} we extend this result for general finite Blaschke
products. Let $R$ be a finite Blaschke product $R$ of degree at least two.
We showed that the quotient algebra $\mathcal{OC}_{R}$
is isomorphic to a certain Cuntz-Pimsner algebra and
there is a relation between the quotient algebra $\mathcal{OC}_R$ and
the C$^*$-algebra $\mathcal{O}_R(J_R)$. In general, two C$^*$-algebras
$\mathcal{OC}_R$ and $\mathcal{O}_R(J_R)$ are slightly different.

In this paper we give a relation
between a C$^*$-algebra containing a composition operator
and the C$^*$-algebra $\mathcal{O}_R(J_R)$
for a general rational function $R$ of degree at least two.
In the above studies we deal with composition operators on the Hardy space
$H^2(\mathbb{D})$, while
we consider composition operators on $L^2$ spaces in this case.
Composition operators on $L^2$ spaces has been studied
by many authors (see for example \cite{SM}).
Let $(\Omega, \mathcal{F}, \mu)$ be a measure space and
let $\varphi$ a non-singular transformation on $\Omega$.
We define a measurable function by
$C_\varphi f = f \circ \varphi$ for $f \in L^2(\Omega, \mathcal{F}, \mu)$.
If $C_\varphi$ is bounded operator on $L^2 (\Omega, \mathcal{F}, \mu)$,
we call $C_\varphi$ the composition operator with $\varphi$.

Let $R$ be a rational function of degree at least two.
We consider the Julia set $J_R$ of $R$, the Borel $\sigma$-algebra
$\mathcal{B}(J_R)$ on $J_R$ and the Lyubich measure $\mu^L$ of $R$.
Let us denote by $\mathcal{MC}_R$ the C$^*$-algebra generated by
multiplication operators $M_a$ for $a \in C(J_R)$
and the composition operator $C_R$ on
$L^2 (J_R, \mathcal{B}(J_R), \mu^L)$.
We regard the C$^*$-algebra $\mathcal{MC}_R$ and multiplication
operators as replacements of
Toeplitz-composition C$^*$-algbras and Toeplitz operators, respectively.
We prove that the C$^*$-algebra $\mathcal{MC}_R$ is
isomorphic to the C$^*$-algebra $\mathcal{O}_R(J_R)$ associated
with the complex dynamical system.

There are two important points to prove this theorem.
First one is
to analyze operators of the form $C_R ^* M_a C_R$ for $a \in C(J_R)$.
We now consider a more general case.
Let $(\Omega, \mathcal{F}, \mu)$ be a finite measure space and $\varphi$
is a non-singular transformation. If $C_\varphi$ is bounded, then we have
$C_\varphi ^* M_a C_\varphi = M_{\mathcal{L}_\varphi(a)}$ for $a \in L^\infty
(\Omega, \mathcal{F}, \mu)$, where $\mathcal{L}_\varphi$ is the
Frobenius-Perron operator for $\varphi$. This is an extension of
covariant relations considered by Exel and Vershik \cite{EV}.
Moreover similar relations have been studied on the Hardy space
$H^2(\mathbb{D})$.
Let $\varphi$ be an inner function on $\mathbb{D}$.
Jury showed a covariant relation
$C_\varphi ^* T_a C_\varphi = T_{A_\varphi (a)}$ for
$a \in L^\infty(\mathbb{T})$,
where $A_\varphi$ is the Aleksandrov operator.

Second important point is an anaysis based on bases of Hilbert bimodules.
In \cite{H} and \cite{HW}, a Toeplitz-composition C$^*$-algebra
for a finite Blaschke product $R$
is isomorphic to a certain Cuntz-Pimsner algebra of a Hilbert bimodule $X_R$,
using a finite basis of $X_R$.
Let $R$ be a rational function of degree at least two.
The C$^*$-algebra $\mathcal{O}_R(J_R)$ associated with complex dynamical
system is defined as a Cuntz-Pimsner algebra of a Hilbert bimodule $Y$.
Unlike the cases of \cite{H} and \cite{HW}, the Hilbert bimodule $Y$ does not always have a {\it finite} basis.
Kajiwara \cite{Kaj}, however, constructed a concrete {\it countable} basis of
$Y$.
Thanks to this basis, we can prove the desired theorem.

\section{Covariant relations}

Let $(\Omega, \mathcal{F}, \mu)$ be a measure space and
let $\varphi: \Omega \to \Omega$ be a measurable transformation.
Set $ \varphi_* \mu (E) = \mu (\varphi^{-1} (E))$ for $E \in \mathcal{F}$.
Then $\varphi_* \mu$ is a measure on $\Omega$.
The measurable transformation $\varphi : \Omega \to \Omega$ is said to be
{\it non-singular} if $\varphi_* \mu(E) = 0$ whenever $\mu (E) = 0$ for
$E \in \mathcal{F}$. If $\varphi$ is non-singular, then
$\varphi_* \mu$ is absolutely continuous with respect to $\mu$.
When $\mu$ is $\sigma$-finite, we denote by $h_\varphi$ the
Radon-Nikodym derivative $\frac{d \varphi_* \mu}{d \mu}$.

Let $1 \leq p \leq \infty$.
We shall define the composition operator on $L^p(\Omega, \mathcal{F}, \mu)$.
Every non-singular transformation $\varphi: \Omega \to \Omega$ induces a linear
operator $C_\varphi$ from $L^Fp(\Omega, \mathcal{F}, \mu)$ to the linear space
of all measurable functions on $(\Omega, \mathcal{F}, \mu)$ defined as
$C_\varphi f = f \circ \varphi$ for $f \in L^p(\Omega, \mathcal{F}, \mu)$.
If $C_\varphi: L^p(\Omega, \mathcal{F}, \mu) \to L^p(\Omega, \mathcal{F}, \mu)$
is bounded, it is called a {\it composition operator} on
$L^p(\Omega, \mathcal{F}, \mu)$ induced by $\varphi$.
Let $(\Omega, \mathcal{F}, \mu)$ be $\sigma$-finite.
For $1 \leq p < \infty$, $C_\varphi$ is bounded on
$L^p(\Omega, \mathcal{F}, \mu)$ if and only if the Radon-Nikodym derivative
$h_\varphi$ is bounded (see for example \cite[Theorem 2.1.1]{SM}).
If $C_\varphi$ is bounded on $L^p(\Omega, \mathcal{F}, \mu)$ for some
$1 \leq p < \infty$, then
$C_\varphi$ is bounded on $L^p(\Omega, \mathcal{F}, \mu)$ for any
$1 \leq p < \infty$ since $h_\varphi$ is independent of $p$.
For $p=\infty$, $C_\varphi$ is bounded on
$L^\infty (\Omega, \mathcal{F}, \mu)$ for any non-singular transformation.

\begin{defn}
Let $(\Omega, \mathcal{F}, \mu)$ be a $\sigma$-finite measure space,
let $\varphi: \Omega \to \Omega$ be a non-singular transformation and
let $f \in L^1(\Omega, \mathcal{F}, \mu)$.
We define $\nu_{\varphi, f}$ by
\[
  \nu_{\varphi, f} (E) = \int_{\varphi^{-1}(E)} f d \mu
\]
for $E \in \mathcal{F}$.
Then $\nu_{\varphi, f}$ is an absolutely continuous measure with respect
to $\mu$. By the Radon-Nikodym theorem, there exists $\mathcal{L}_\varphi (f)
\in L^1(\Omega, \mathcal{F}, \mu)$ such that
\[
  \int_E \mathcal{L}_\varphi (f) d \mu = \int_{\varphi^{-1}(E)} f d \mu
\]
for $E \in \mathcal{F}$.
We can regard $\mathcal{L}_\varphi$ as a bounded operator on
$L^1(\Omega, \mathcal{F}, \mu)$ (see for example \cite[Proposition 3.1.1]{LM}).
We call $\mathcal{L}_\varphi$ the {\it Frobenius-Perron operator}.
\end{defn}

\begin{lem} \label{lem:adjoint}
Let $(\Omega, \mathcal{F}, \mu)$ be a finite measure space
and let $\varphi: \Omega \to \Omega$ be a non-singular transformation.
Suppose that $C_\varphi : L^1 (\Omega, \mathcal{F}, \mu) \to
L^1 (\Omega, \mathcal{F}, \mu)$ is bounded. Then the restriction
$\mathcal{L}_\varphi |_{L^\infty (\Omega, \mathcal{F}, \mu)}$
is a bounded operator on $L^ \infty (\Omega, \mathcal{F}, \mu)$
and $C_\varphi ^* = \mathcal{L}_\varphi |_{L^\infty (\Omega, \mathcal{F}, \mu)}$.
\end{lem}

\begin{proof}
Let $f \in L^\infty (\Omega, \mathcal{F}, \mu)$.
First we shall show $\mathcal{L}_\varphi (f) \in
L^\infty (\Omega, \mathcal{F}, \mu)$.
There exists $M>0$ such that $|f| \leq M$.
It follows from \cite[Proposition 3.1.1]{LM} that
$|\mathcal{L}_\varphi (f) | \leq \mathcal{L}_\varphi (|f|) \leq
 M \mathcal{L}_\varphi (1)$.
Since $\mathcal{L}_\varphi (1) = h_\varphi$ and $C_\varphi$ is bounded
on $L^1 (\Omega, \mathcal{F}, \mu)$,
we have $\mathcal{L}_\varphi(1) \in L^\infty (\Omega, \mathcal{F}, \mu)$.
Hence $\mathcal{L}_\varphi (f) \in L^\infty (\Omega, \mathcal{F}, \mu)$.

By the definition of $\mathcal{L}_\varphi$, we have
\[
  \int_\Omega \chi_E \mathcal{L}_\varphi (f) d \mu
  = \int_\Omega \chi_{\varphi^{-1}(E)} f d \mu
  = \int_\Omega (C_\varphi \chi_E) f d \mu
\]
for $E \in \mathcal{F}$, where $\chi_E$ and $\chi_{\varphi^{-1}(E)}$
are characteristic functions on $E$ and $\varphi^{-1}(E)$
respectively.
Since $C_\varphi$ is bounded on $L^1(\Omega, \mathcal{F}, \mu)$
and the set of integrable simple functions is dense in
$L^1(\Omega, \mathcal{F}, \mu)$, the restriction map
$\mathcal{L}_\varphi |_{L^\infty (\Omega, \mathcal{F}, \mu)}$
is bounded on $L^\infty (\Omega, \mathcal{F}, \mu)$
and $C_\varphi ^* =
\mathcal{L}_\varphi |_{L^\infty (\Omega, \mathcal{F}, \mu)}$.
\end{proof}

Let $(\Omega, \mathcal{F}, \mu)$ be a finite measure space
and $\varphi: \Omega \to \Omega$ a non-singular transformation.
We consider the restriction of $\mathcal{L}_\varphi$ to
$L^\infty (\Omega, \mathcal{F}, \mu)$.
From now on, we use the same notation $\mathcal{L}_\varphi$
if no confusion can arise.

For $a \in L^\infty (\Omega, \mathcal{F}, \mu)$, we define
the multiplication operator $M_a$ on
$L^2 (\Omega, \mathcal{F}, \mu)$ by
$M_a f = a f$ for $f \in L^2 (\Omega, \mathcal{F}, \mu)$.
We show the following covariant relation.

\begin{prop} \label{prop:covariant}
Let $(\Omega, \mathcal{F}, \mu)$ be a finite measure space
and let $\varphi: \Omega \to \Omega$ be a non-singular transformation. If
$C_\varphi: L^2 (\Omega, \mathcal{F}, \mu) \to L^2 (\Omega, \mathcal{F}, \mu)
$ is bounded, then we have
\[
  C_\varphi ^* M_a C_\varphi = M_{\mathcal{L}_\varphi (a)}
\]
for $a \in L^\infty (\Omega, \mathcal{F}, \mu)$.
\end{prop}

\begin{proof}
For $f, g \in L^2 (\Omega, \mathcal{F}, \mu)$,
we have
\begin{align*}
  \langle C_\varphi ^* M_a C_\varphi f, g \rangle
  &= \langle M_a C_\varphi f,  C_\varphi g \rangle
  = \int_\Omega a (f \circ \varphi) 
    \overline {(g \circ \varphi)} d \mu \\
  &= \int_\Omega a C_\varphi (f \overline{g}) d \mu
  = \int_\Omega \mathcal{L}_\varphi (a) f \overline{g} d \mu \\
  &= \langle M_{\mathcal{L}_\varphi (a)} f, g \rangle
\end{align*}
by Lemma \ref{lem:adjoint}, where
$C_\varphi$ is also regarded as the composition operator
on $L^1(\Omega, \mathcal{F}, \mu)$.
\end{proof}

\section{C$^*$-algebras associated with complex dynamical systems}

We recall the construction of Cuntz-Pimsner algebras \cite{Pi} (see also
\cite{K}). 
Let $A$ be a C$^*$-algebra and let $X$ be a right Hilbert $A$-module.
A sequence $\{ u_i \}_{i=1} ^\infty$
of $X$ is called a {\it countable basis} of X if
$\xi = \sum_{i=1} ^\infty u_i \langle u_i, \xi \rangle_A$
for $\xi \in X$, where the right hand side converges in norm.
We denote by $\mathcal{L}(X)$ the C$^*$-algebra of the adjointable bounded operators 
on $X$.  
For $\xi$, $\eta \in X$, the operator $\theta _{\xi,\eta}$
is defined by $\theta _{\xi,\eta}(\zeta) = \xi \langle \eta, \zeta \rangle_A$
for $\zeta \in X$. 
The closure of the linear span of these operators is denoted by $\mathcal{K}(X)$. 
We say that 
$X$ is a {\it Hilbert bimodule} (or {\it C$^*$-correspondence}) 
over $A$ if $X$ is a right Hilbert $A$-module 
with a $*$-homomorphism $\phi : A \rightarrow \mathcal{L}(X)$.
We always assume that $\phi$ is injective. 

A {\it representation} of the Hilbert bimodule $X$
over $A$ on a C$^*$-algebra $D$
is a pair $(\rho, V)$ constituted by a $*$-homomorphism $\rho: A \to D$ and
a linear map $V: X \to D$ satisfying
\[
  \rho(a) V_\xi = V_{\phi(a) \xi}, \quad
  V_\xi ^* V_\eta = \rho( \langle \xi, \eta \rangle_A)
\]
for $a \in A$ and $\xi, \eta \in X$.
It is known that $V_\xi \rho(b) = V_{\xi b}$ follows
automatically (see for example \cite{K}).
We define a $*$-homomorphism $\psi_V : \mathcal{K}(X) \to D$
by $\psi_V ( \theta_{\xi, \eta}) = V_{\xi} V_{\eta}^*$ for $\xi, \eta \in X$
(see for example \cite[Lemma 2.2]{KPW}).
A representation $(\rho, V)$ is said to be {\it covariant} if
$\rho(a) = \psi_V(\phi(a))$ for all $a \in J(X)
:= \phi ^{-1} (\mathcal{K}(X))$.
Suppose the Hilbert bimodule $X$
has a countable basis $\{u_i \}_{i=1} ^\infty$ and $(\rho, V)$ is
a representation of $X$.
Then $(\rho, V)$ is covariant if and only if
$\| \sum_{i=1} ^n \rho(a) V_{u_i} V_{u_i} ^* - \rho(a) \| \to 0$ as
$n \to \infty$ for $a \in J(X)$, since $\{ \sum_{i=1} ^n \theta_{u_i, u_i} \}_{n=1} ^\infty$ is an approximate unit for $\mathcal{K}(X)$.

Let $(i, S)$ be the representation of $X$ which is universal for
all covariant representations. 
The {\it Cuntz-Pimsner algebra} ${\mathcal O}_X$ is 
the C$^*$-algebra generated by $i(a)$ with $a \in A$ and 
$S_{\xi}$ with $\xi \in X$.
We note that $i$ is known to be injective
\cite{Pi} (see also \cite[Proposition 4.11]{K}).
We usually identify $i(a)$ with $a$ in $A$.

Let $R$ be a rational function of degree at least two.
We recall the definition of the C$^*$-algebra
$\mathcal{O}_R (J_R)$. 
Since the Julia set $J_R$ is completely invariant under $R$, that is,
$R(J_R) = J_R = R^{-1}(J_R)$, we can consider the restriction 
$R|_{J_R} : J_R \rightarrow J_R$.
Let $A = C(J_R)$ and $Y = C(\graph R|_{J_R})$,
where $\graph R|_{J_R} = \{(z,w) \in J_R \times J_R \, \, | \, \, w = R(z)\} $ 
is the graph of $R|_{J_R}$.
We denote by $e_R(z)$ the branch index of $R$ at $z$.
Then $Y$ is an $A$-$A$ bimodule over $A$ by 
\[
  (a\cdot f \cdot b)(z,w) = a(z) f(z,w) b(w),\quad a,b \in A, \, 
  f \in Y.
\]
We define an $A$-valued inner product $\langle \ , \ \rangle_A$ on $Y$ by 
\[
  \langle f, g \rangle_A(w) = \sum _{z \in R^{-1}(w)} e_R(z)
  \overline{f(z,w)}g(z,w),
  \quad f, g \in Y, \, w \in J_R.
\]
Then $Y$ is a Hilbert bimodule over $A$. 
The C$^*$-algebra ${\mathcal O}_R(J_R)$
is defined as the Cuntz-Pimsner algebra of the Hilbert bimodule 
$Y = C(\graph R|_{J_R})$ over $A = C(J_R)$.

\section{Main theorem}

Let $R$ be a rational function.
We define the backward orbit
$O^-(w)$ of $w \in \hat{\mathbb{C}}$ by
\[
   O ^- (w) = \{ z \in \hat{\mathbb{C}} \ | \ R^{\circ m} (z) = w
   \ \text{for some non-negative integer} \ m \}.
\]
A point $w$ in $\hat{\mathbb{C}}$ is an {\it exceptional point}
for $R$ if the backward orbit $O^- (w)$ of $w$ is finite.
We denote by $E_R$ the set of exceptional points.

\begin{defn}[Freire-Lopes-Ma\~{n}\'{e} \cite{FLM}, Lyubich \cite{L}]
Let $R$ be a rational function and $n = {\rm deg} R$.
Let $\delta_z$ be the Dirac measure at $z \in \hat{\mathbb{C}}$.
For $w \in \hat{\mathbb{C}} \smallsetminus E_R$ and $m \in \mathbb{N}$,
we define a probability measure $\mu_m ^w$ on $\hat{\mathbb{C}}$ by
\[
  \mu_m ^w = \frac{1}{n^m} \sum_{z \in (R^{\circ m})^{-1} (w)}
  e_{R^{\circ m}}(z) \delta_z.
\]
The sequence $\{ \mu_m ^w \}_{m=1} ^\infty$ converges weakly to a probability
measure $\mu^L$, which is called the {\it Lyubich measure} of $R$.
The measure $\mu^L$ is independent of the choice of $w \in
\hat{\mathbb{C}} \smallsetminus E_R$.
\end{defn}

Let $R$ be a rational function of degree at least two.
We will denote by $\mathcal{B}(J_R)$ the Borel $\sigma$-algebra on the Julia
set $J_R$.
In this section we consider the finite measure space
$(J_R, \mathcal{B}(J_R), \mu^{L})$.
It is known that the support of 
the Lyubich measure $\mu^L$ is the Julia set $J_R$.
Moreover the Lyubich measure $\mu^L$ is regular on the Julia set $J_R$ and
a invariant measure with respect to $R$,
that is, $\mu^L (E) = \mu^L (R^{-1} (E))$ for $E \in \mathcal{B}(J_R)$.
Thus the composition operator $C_R$ on $L^2 (J_R, \mathcal{B}(J_R), \mu^{L})$
is an isometry.

\begin{defn}
For a rational function $R$ of degree at least two,
we denote by $\mathcal{MC}_R$ the C$^*$-algebra
generated by all multiplication operators by
continuous functions in $C(J_R)$
and the composition operator $C_R$ on $L^2(J_R, \mathcal{B}(J_R), \mu^{L})$.
\end{defn}

Let a rational function $R$ of degree at least two.
In this section we shall show that the C$^*$-algebra $\mathcal{MC}_R$
is isomorphic to the C$^*$-algebra $\mathcal{O}_R (J_R)$.
First we give a concrete expression of the restriction of
$\mathcal{L}_R$ to $C(J_R)$.
This result immediately follows from \cite{L} and
Lemma \ref{lem:adjoint}.

\begin{prop}[Lyubich {\cite[Lemma, p.366]{L}}]
Let $R$ be a rational function of degree $n$ at least two.
Then $\mathcal{L}_R : C(J_R) \to C(J_R)$ and
\[
  (\mathcal{L}_R (a)) (w) = \frac{1}{n} \sum_{z \in R^{-1} (w)} e_R (z) a(z),
  \quad w \in J_R
\]
for $a \in C(J_R)$.
\end{prop}

Let $X = C(J_R)$ and $n = \deg R$.
Then $X$ is an $A$-$A$ bimodule over $A$ by 
\[
   (a \cdot \xi \cdot b) (z) = a(z) \xi(z) b(R(z)) \quad a,b \in A, \, 
  \xi \in X.
\]
We define an $A$-valued inner product $\langle \ , \ \rangle_A$ on $X$ by
\[
  \langle \xi, \eta \rangle _A (w) = \frac{1}{n} \sum_{z \in R^{-1} (w)}
  e_R (z) \overline{\xi(z)} \eta(z)
  \, \,
  \left (\, = (\mathcal{L}_R (\overline{\xi} \eta))(w) \, \right ), \quad
  \xi, \eta \in X.
\]
Then $X$ is a Hilbert bimodule over $A$.
Put $\| \xi \|_2 = \| \langle \xi, \xi \rangle_A \|_\infty ^{1/2}$
for $\xi \in X$,
where $\| \ \|_\infty$ is the sup norm on $J_R$.
It is easy to see that $X$ is isomorphic to $Y$ as Hilbert bimodules over $A$.
Hence the C$^*$-algebra $\mathcal{O}_R(J_R)$ is isomorphic to
the Cuntz-Pimsner algebra $\mathcal{O}_X$ constructed from $X$.

We need some analyses based on bases of the Hilbert bimodule $X$
to show an equation containing the composition operator $C_R$
and multiplication operators.

\begin{lem} \label{lem:key_lem}
Let $u_1, \dots, u_N \in X$.
Then
\[
  \sum_{i=1} ^N M_{u_i} C_R C_R ^* M_{u_i} ^* a
  = \sum_{i=1} ^N u_i \cdot \langle u_i , a \rangle_A
\]
for $a \in A$.
\end{lem}

\begin{proof}
Since $a = M_a C_R 1$, we have
\begin{align*}
\sum_{i=1} ^N M_{u_i} C_R C_R ^* M_{u_i} ^* a
&= \sum_{i=1} ^N M_{u_i} C_R C_R ^* M_{u_i} ^* M_a C_R 1 \\
&= \sum_{i=1} ^N M_{u_i} C_R C_R ^* M_{\overline u_i a} C_R 1 \\
&= \sum_{i=1} ^N M_{u_i} C_R M_{\mathcal{L}_R(\overline u_i a)} 1
   \quad \quad \text{by Proposition \ref{prop:covariant}} \\
&= \sum_{i=1} ^N M_{u_i} M_{\mathcal{L}_R(\overline u_i a) \circ R} C_R 1 \\
&= \sum_{i=1} ^N u_i \mathcal{L}_R(\overline u_i a) \circ R \\
&= \sum_{i=1} ^N u_i \cdot \langle u_i , a \rangle_A,
\end{align*}
which completes the proof.
\end{proof}

\begin{lem} \label{lem:norm}
Let $\{ u_i \}_{i=1} ^\infty$ be a countable basis of $X$.
Then
\[
  0 \leq \sum_{i=1} ^N M_{u_i} C_R C_R ^* M_{u_i} ^*  \leq I. 
\]
\end{lem}

\begin{proof}
Set $T_N := \sum_{i=1} ^N M_{u_i} C_R C_R ^* M_{u_i} ^*$.
It is clear that $T_N$ is a positive operator.
We shall show $T_N \leq I$. By Lemma \ref{lem:key_lem},
\[
\langle T_N f, f \rangle = \int_{J_R} (T_N f) (z)
\overline{f(z)} d \mu ^L (z)
= \int_{J_R} \left( \sum_{i=1} ^N u_i \cdot
\langle u_i, f \rangle_A \right ) (z) \overline{f(z)} d \mu ^L (z)
\]
for $f \in C(J_R)$.
Since $\{ u_i \}_{i=1} ^\infty$ is a countable basis of $X$, 
for $f \in C(J_R)$, we have
$\sum_{i=1} ^N u_i \cdot \langle u_i, f \rangle_A \to f$
with respect to $\| \, \, \|_2$ as $N \to \infty$.
Since the two norms $\| \, \, \|_2$ and $\| \, \, \|_\infty$ are
equivalent (see the proof of \cite[Proposition 2.2]{KW}),
$\sum_{i=1} ^N u_i \cdot \langle u_i, f \rangle_A$ converges to $f$
with respect to $\| \, \, \|_\infty$.
Thus
\[
  \langle T_N f, f \rangle \to \int_{J_R} f(z) \overline{f(z)} d \mu ^L (z) = \langle f, f \rangle \quad {\text as} \, \, \, N \to \infty
\]
for $f \in C(J_R)$.
Therefore $\langle T_N f, f \rangle \leq \langle f, f \rangle$ for $f \in C(J_R)$.
Since the Lyubich measure $\mu^L$ on the Julia set $J_R$
is regular, $C(J_R)$ is dense in $L^2 (J_R, \mathcal{B}(J_R), \mu^{L})$. Hence
we have $T_N \leq I$. This completes the proof.
\end{proof}

Let $\mathcal{B}(R)$ be the set of branched points of a rational function $R$.
We now recall a description of the ideal $J(X)$ of $A$.
By \cite[Proposition 2.5]{KW}, we can write
$J(X) = \{ a \in A \, | \,
a \, \, \text{vanishes on} \, \, \mathcal{B}(R) \}$.
We define a subset $J(X)^0$ of $J(X)$
by $J(X) ^0 = \{ a \in A \, | \,
a \, \, \text{vanishes on} \, \, \mathcal{B}(R) \, \,
\text{and has compact support on } J_R \smallsetminus \mathcal{B}(R) \}$.
Since $\mathcal{B}(R)$ is a finite set (\cite[Corollary 2.7.2]{B}),
$J(X)^0$ is dense in $J(X)$.

\begin{lem} \label{lem:countable_basis}
There exists a countable basis $\{ u_i \}_{i=1} ^\infty$ of $X$
such that
\[
  \sum_{i=1} ^\infty M_a M_{u_i} C_R C_R ^* M_{u_i} ^* = M_a
\]
for $a \in J(X)$.
\end{lem}

\begin{proof}
By \cite[Subsection 3.1]{Kaj},
there exists a countable basis $\{ u_i \}_{i=1} ^\infty$ of $X$
satisfying the following property.
For any $b \in J (X) ^0$, there exists $M>0$  such that
${\rm supp} \, b \cap {\rm supp} \, u_m = \emptyset$ for $m \geq M$.
Since $J (X) ^0$ is dense in $J (X)$,
for any $a \in A$ and any $\varepsilon > 0$, there exists
$b \in J (X) ^0$
such that $\| a - b \| < \varepsilon / 2$.
Let $m \geq M$.
Then by Lemma \ref{lem:key_lem} and $b u_i = 0$ for
$i \geq m$, it follows that
\[
   \sum_{i=1} ^m M_b M_{u_i} C_R C_R ^* M_{u_i} ^* f
   = \sum_{i=1} ^m b u_i \cdot \langle u_i, f \rangle_A
   = \sum_{i=1} ^\infty b u_i \cdot \langle u_i, f \rangle_A
   = bf =M_b f
\]
for $f \in C(J_R)$. Since
$C(J_R)$ is dense in $L^2(J_R, \mathcal{B}(J_R), \mu^L)$, we have
\[
  \sum_{i=1} ^m M_b M_{u_i} C_R C_R ^* M_{u_i} ^* = M_b.
\]
From Lemma \ref{lem:norm} it follows that
\begin{align*}
\left \| \sum_{i=1} ^m M_a M_{u_i} C_R C_R ^* M_{u_i} ^* - M_a \right \|
& \leq
\left \| \sum_{i=1} ^m M_a M_{u_i} C_R C_R ^* M_{u_i} ^*
- \sum_{i=1} ^m M_b M_{u_i} C_R C_R ^* M_{u_i} ^* \right \| \\
& \quad \quad +
\left \| \sum_{i=1} ^m M_b M_{u_i} C_R C_R ^* M_{u_i} ^* - M_b
\right \| + \| M_b - M_a \| \\
& \leq  \| M_a - M_b \| \, \left \| \sum_{i=1} ^m M_{u_i} C_R C_R ^* M_{u_i} ^*
\right \| + \| M_a - M_b \| \\
& < \frac{\varepsilon}{2} + \frac{\varepsilon}{2} = \varepsilon,
\end{align*}
which completes the proof.
\end{proof}

The following theorem is the main result of the paper.

\begin{thm}
Let $R$ be a rational function of degree at least two.
Then $\mathcal{MC}_R$ is isomorphic to
$\mathcal{O}_R(J_R)$.
\end{thm}

\begin{proof}
Put $\rho(a) = M_a$ and $V_\xi = M_\xi C_R$ for $a \in A$ and $\xi \in X$.
Then we have
\[
  \rho(a) V_\xi = M_a M_\xi C_R = M_{a \xi} C_R =V_{a \cdot \xi}
\]
and
\[
  V_\xi ^* V_\eta = C_R ^* M_\xi ^* M_\eta C_R 
  = C_R ^* M_{\overline{\xi} \eta} C_R
  = M_{\mathcal{L}_R (\overline{\xi} \eta)}
  = \rho (\mathcal{L}_R (\overline{\xi} \eta))
  = \rho ( \langle \xi, \eta \rangle_A)
\]
for $a \in A$ and $\xi, \eta \in X$
by Proposition \ref{prop:covariant}.
Let $\{ u_i \} _{i=1} ^\infty$ be a countable basis of $X$.
Then, applying Lemma \ref{lem:countable_basis},
\[
  \sum_{i=1} ^\infty \rho(a) V_{u_i} V_{u_i} ^*
  = \sum_{i=1} ^\infty M_a M_{u_i} C_R C_R ^* M_{u_i} ^*
  = M_a = \rho(a)
\]
for $a \in J (X)$.
Since the support of the Lyubich measure $\mu^L$ is
the Julia set $J_R$, the $*$-homomorphism $\rho$ is injective.
By the universality and the simplicity of $\mathcal{O}_R(J_R)$
(\cite[Theorem 3.8]{KW}),
the C$^*$-algebra $\mathcal{MC}_R$ is isomorphic to
$\mathcal{O}_R(J_R)$.
\end{proof}

\begin{ack}
The author wishes to express his thanks to Professor Hiroyuki Takagi
for several helpful comments concerning to composition operators.
\end{ack}

\end{document}